\documentclass[12pt,reqno]{article}

\usepackage[usenames]{color}
\usepackage{amssymb}
\usepackage{graphicx}
\usepackage{amscd}
\usepackage{hyperref}

\definecolor{webgreen}{rgb}{0,.5,0}
\definecolor{webbrown}{rgb}{.6,0,0}

\usepackage{color}
\usepackage{fullpage}
\usepackage{float}

\usepackage{graphics,amsmath,amssymb}
\usepackage{amsthm}
\usepackage{amsfonts}
\usepackage{latexsym}
\usepackage{epsf}

\usepackage{tikz}
\usepackage{forest}

\newcommand{\seqnum}[1]{\href{http://oeis.org/#1}{\underline{#1}}}

\begin{document}


\theoremstyle{plain}
\newtheorem{theorem}{Theorem}
\newtheorem{corollary}[theorem]{Corollary}
\newtheorem{lemma}[theorem]{Lemma}
\newtheorem{proposition}[theorem]{Proposition}

\theoremstyle{definition}
\newtheorem{definition}[theorem]{Definition}
\newtheorem{example}[theorem]{Example}
\newtheorem{conjecture}[theorem]{Conjecture}
\newtheorem{identity}{Identity}

\theoremstyle{remark}
\newtheorem{remark}[theorem]{Remark}

\begin{center}
\vskip 1cm{\LARGE\bf 
On Base 3/2 and its Sequences
}
\vskip 1cm
\large
Ben Chen, Richard Chen, Joshua Guo, Shane Lee, Neil Malur, Nastia Polina, Poonam Sahoo, Anuj Sakarda, Nathan Sheffield, Armaan Tipirneni\\ 
PRIMES STEP, Room 2-231C\\
Math Dept., MIT\\
77 Mass. Ave \\
Cambridge,  MA 02139  \\
USA \\
\href{mailto: primes.step@gmail.com}{\tt primes.step@gmail.com}
\ \\
Tanya Khovanova \\
Department of Mathematics\\ 
MIT\\
77 Mass. Ave \\
Cambridge,  MA 02139  \\
USA \\
\href{mailto: tanyakh@yahoo.com}{\tt tanyakh@yahoo.com}
\end{center}

\vskip .2 in

\begin{abstract}
We discuss properties of integers in base 3/2. We also introduce many new sequences related to base 3/2. Some sequences discuss patterns related to integers in base 3/2. Other sequence are analogues of famous base-10 sequences: we discuss powers of 3 and 2, Look-and-say, and sorted and reverse sorted Fibonaccis. The eventual behavior of sorted and reverse sorted Fibs leads to special Pinocchio and Oihcconip sequences respectively.
\end{abstract}

\section{Introduction}

What is base 3/2? How does one even think about a fractional base anyway? Our readers will be familiar of course with base 10, or decimal, but there are many uses for other bases, such as base 2, 12, and 60. Base 2, or binary, is useful because there are only two states for each place value, meaning that it can be represented easily by a series of transistors that are either on or off, and forms the basis for machine language. Base 12 and 60 are useful because they have a large number of factors and hence can be divided nicely into smaller increments. We use these bases to partition time. 

One way of thinking about how integer bases such as these work, invented by James Tanton, is the idea of exploding dots \cite{JT}. This idea allows a natural extension into fractional bases. 

Our main base is base 3/2. We explain exploding dots and base 3/2 in detail in Section~\ref{sec:expdots}.

We start by observing how integers are represented in base 3/2. The integers use 0, 1, and 2 as digits. The beginnings of integers are very restricted, but the endings are not. We discuss this in Sections~\ref{sec:patterns}.

Even integers can be naturally placed into vertices of a tree with alternating branches. The base 3/2 representation of an even integer can be read from the tree as explained in Section~\ref{sec:begin}. We devote some time to studying properties of the largest and smallest integers with a given number of digits in base 3/2. We also discuss the connection of base 3/2 to a greedy partition of non-negative integers into subsequences not containing a 3-term arithmetic progression.

In Section~\ref{sec:divisibility} we briefly talk about divisibility properties. We show that divisibility by 5 in base 3/2 is similar to divisibility by 11 in base 10.

By using exploding dots to define base 3/2, we can now begin to examine how sequences behave in new, less studied environments. Our goal in this paper is to briefly explore the behavior of various sequences in base 3/2. We study in particular sequences which depend on their positional representation, so that a different base causes them to behave in new ways, as we discuss in Section~\ref{sec:sequences}.

We notice that powers of 2 and 3 exhibit awesome properties in base 3/2. Integer $3^n$ in base 3/2 is $2^n$ followed by $n$ zeros.

The Look-and-say sequence in base 3/2 begins similarly to base 10, but differs after the first five terms.

Fibonacci numbers are incredibly beautiful, but because they are not dependent on base, their base 3/2 expression might not be particularly  interesting. By tweaking the Fibonacci sequence using sorting of the digits as one of the steps, we are able to create sequences related to base representations that are inspired by the Fibonacci sequence. We study the eventual behavior of such sequences.

\section{Exploding dots and base 3/2}\label{sec:expdots}

Essentially, exploding dots is a machine made up of boxes with rules to describe what happens when you have a certain number of dots in a box \cite{JT}. In base 10, whenever there is a group of 10 dots in one box, they explode into 1 dot in the next box up. Similarly, we could use this to describe binary by having 2 dots explode into 1, and so on for any base. 

For example, to write 11 in base 3, we would first have 11 dots in the rightmost box as in Figure~\ref{fig:11base3step1}.
\begin{figure}[htb]
    \centering
    \scalebox{-1}[1]{ \includegraphics[scale=0.4]{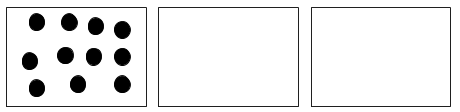}}
    \caption{11 base 3: step 1.}
    \label{fig:11base3step1}
\end{figure}

Then each group of 3 dots in the rightmost box would explode, and one dot per group will appear in the box to the left as in Figure~\ref{fig:11base3step2}.
\begin{figure}[htb]
    \centering
    \scalebox{-1}[1]{ \includegraphics[scale=0.4]{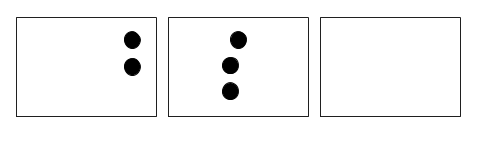}}
    \caption{11 base 3: step 2.}
    \label{fig:11base3step2}
\end{figure}

Finally, the three dots in the second box would explode into 1 dot to its left, as shown in Figure~\ref{fig:11base3step3}. 
\begin{figure}[htb]
    \centering
    \scalebox{-1}[1]{ \includegraphics[scale=0.4]{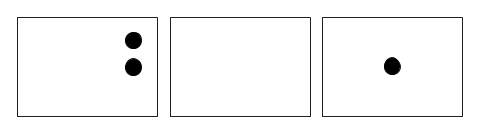}}
    \caption{11 base 3: step 3.}
    \label{fig:11base3step3}
\end{figure}

By reading the number of dots from left to right, 11 in base 3 is 102.

But the interesting thing here is that there is no reason this model should be exclusive to integer bases \cite{JT}. Suppose, instead, our rule is that 3 dots explode into 2 dots in the next box. To represent 11 in this base, we use a similar process, shown in Figure~\ref{11base3Over2.png}.
\begin{figure}[htb]
    \centering
    \scalebox{-1}[1]{ \includegraphics[scale=0.4]{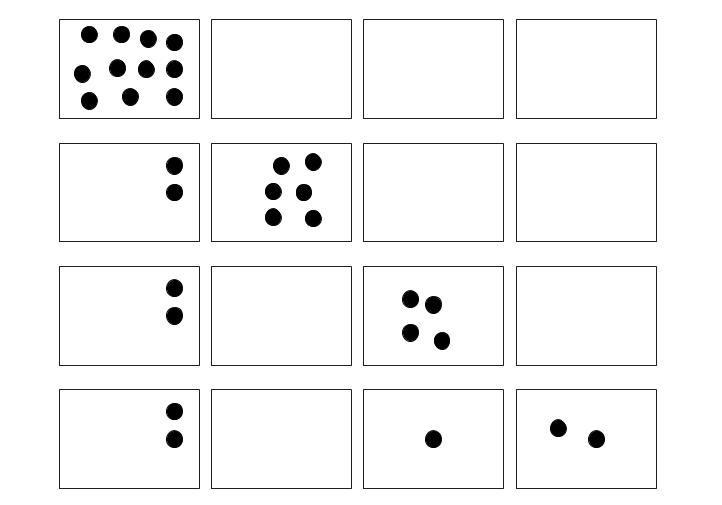}}
    \caption{11 base 3/2.}
    \label{11base3Over2.png}
\end{figure}

Each group of three explodes into two in the next box. Using this system, 11 is 2102.

This representation behaves quite a bit like base 3/2. The rightmost box represents $(\frac{3}{2})^0$, the next $(\frac{3}{2})^1$, then $(\frac{3}{2})^2$, and so on. Number $2\cdot (\frac{3}{2})^3+1 \cdot (\frac{3}{2})^2+0 \cdot (\frac{3}{2})^1+2 \cdot(\frac{3}{2})^0$ is indeed equal to 11. We can use this just like any other base to represent numbers.

\section{Patterns in integers written in base 3/2.}\label{sec:patterns}

The first few non-negative integers are expressed in base 3/2 as:
\[0,\ 1,\ 2,\ 20,\ 21,\ 22,\ 210,\ 211,\ 212,\ 2100,\ 2101,\ 2102,\ 2120, \ldots.\]

This is sequence \seqnum{A024629} in the database. We start by studying the patterns at the beginning of the integers.

\subsection{Beginning digits}\label{sec:begin}

Here are several properties of integers written in base 3/2 that are easy to observe \cite{CM}:

\begin{enumerate}
\item From 2 onwards all integers start with 2. 
\item From 6, or 210 in base 3/2, onwards all integers start with 21. 
\item The only integer with more than one digit that has all of the same digit is 5, which is 22 in base 3/2.
\item From 8 onward, the third digit changes between 0 and 2. 
\item No integer other than 7 starts with 211.
\end{enumerate}

These properties are easy to prove. For example, the first property is true because each carry adds 2 \cite{JT}. The second and the fourth properties can be proven with a similar argument. The third/fifth  property follows from the second/fourth property, respectively. 

Notice that a prefix of any integer written in base 3/2 is an even integer. For example, integer 32 in base 3/2 is 212022. After removing the last digit we get 21202, which is 20 in base 3/2.

\begin{lemma}
Removing the last digit of integer $n$ in base 3/2 produces integer $2 \cdot \lfloor \frac{n}{3}\rfloor$.
\end{lemma}

\begin{proof}
Exactly $2 \cdot \lfloor \frac{n}{3}\rfloor$ dots move to the left after all explosions in the units digit.
\end{proof}

In the example above $2 \cdot \lfloor \frac{32}{3}\rfloor = 20$.

This observation that, in base 3/2 a proper prefix of an integer is an even integer allows us to arrange even integers in a tree. 

At each vertex we have a digit and an even integer value as a subscript. The tree is pictured in Figure~\ref{fig:tree}. Even integers are written from top to bottom left to right. If a vertex is marked with integer $x$, then we concatenate the digits on the path from the root to $x$ to get the representation of $x$ in base 3/2.

The tree is built recursively starting with integer 2, which is 2 in base 3/2. The recursive rule is: If the integer at the vertex $x$ divided by 2 is odd, then exactly one edge goes out, and the resulting node is labeled 1. The corresponding integer is $3x/2+1$. Otherwise, if the half is even, then we draw two edges and add 0 for $3x/2$ and add 2  for $3x/2+2$ to them. Note that the nodes branch alternatively corresponding to their integer values.

\begin{figure}[h]\label{fig:tree}
\centering
\scalebox{1.3}{
\begin{forest}
[$2_{2}$[$1_{4}$[$0_{6}$[$1_{10}$[$1_{16}$[$0_{24}$[$0_{36}$][$2_{38}$]][$2_{26}$[$1_{40}$]]]]][$2_{8}$[$0_{12}$[$0_{18}$[$1_{28}$[$0_{42}$][$2_{44}$]]][$2_{20}$[$0_{30}$[$1_{46}$]][$2_{32}$[$0_{48}$][$2_{50}$]]]][$2_{14}$[$1_{22}$[$1_{34}$[$1_{52}$]]]]]]]
\end{forest}
}
\caption{The tree of even integers}
\end{figure}
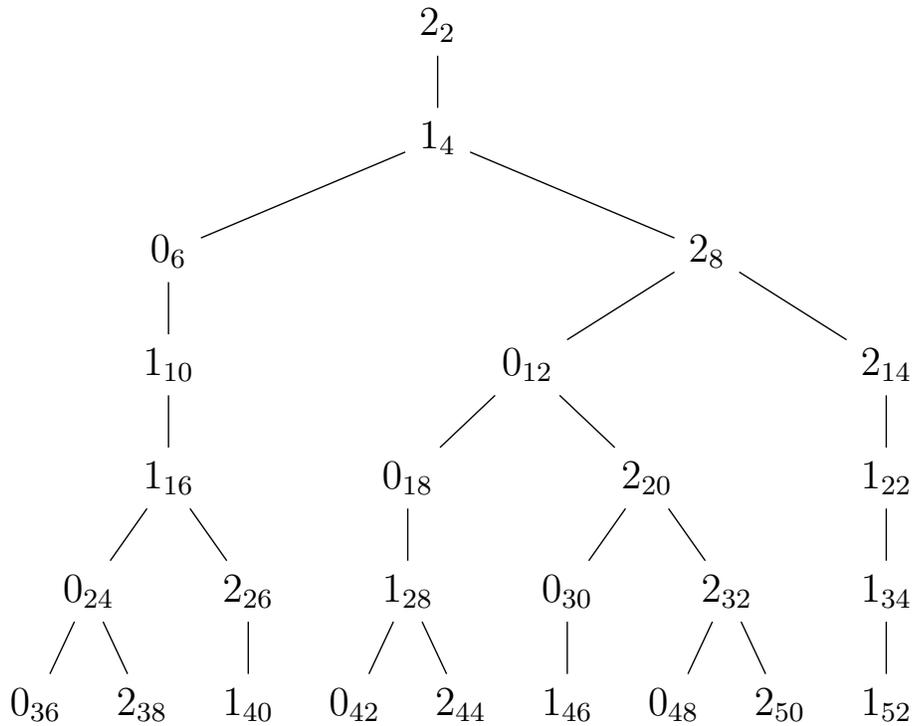

Note that the tree is mentioned in \seqnum{A005428} in a different context. The sequence \seqnum{A005428} is described as the number of nodes at level $n$ of a planted binary tree with alternating branching and non-branching nodes. As we saw above our tree branches out when the corresponding even number is divisible by 4. As even numbers divisible by 4 alternate with even numbers that are not divisible by 4, our tree is the same tree as in the sequence. James Propp and Glen Whitney directed us to the tree \cite{JP, GW}.

Each vertex on level $k$ of the tree, counting from top to bottom, corresponds to an integer that has $k$ digits in its base 3/2 representation. See sequence \seqnum{A246435} that represents the number of digits of integers in base 3/2.

The number of nodes at distance $k-1$ from the root is the number of even numbers with $k$ digits. The corresponding sequence is:
\[1,\ 1,\ 2,\ 3,\ 4,\ 6,\ 9,\ 14,\ 21,\ 31,\ 47,\ 70, \ldots.\]

This is sequence \seqnum{A005428}: $a(n) = \lceil 1 + \textrm{sum of preceding terms} /2\rceil$. It is also a shifted sequence \seqnum{A073941}.

The number of all integers and the number of even integers of a given length in base 3/2 are related in the following manner.

\begin{lemma}
Given a prefix $p$, thrice the number of even integers of length $k$ with prefix $p$ is the number of all integers of length $k+1$ with prefix $p$.
\end{lemma}

\begin{proof}
Any integer can be written as an even integer with one of the digits 0, 1, and 2 attached at the end.
\end{proof}

It follows that the number of integers with $k+1$ digits in base 3/2 is thrice the number of even integers with $k$ digits. The following sequence, \seqnum{A081848}, describes the total number of integers of length $k$ in base 3/2:
\[\seqnum{A081848}(n) = 3,\ 3,\ 3,\ 6,\ 9,\ 12,\ 18,\ 27,\ 42,\ 63,\ 93,\ 141,\ 210,\ \ldots.\]

\begin{lemma}
\seqnum{A081848}$(n) = 3\lceil \textrm{sum of preceding terms}/2 \rceil - \textrm{sum of preceding terms}$.
\end{lemma}

\begin{proof}
Denote the sum of the preceding terms $S$, that is, $S$ is the total number of non-negative integers with less than $k$ digits in base 3/2. Also $S$ is the smallest integer with $k$ digits (remember that we are counting 0). The smallest integer with $k+1$ digits is $3\lceil S/2\rceil$. The total number of $k$-digit base 3/2 integers is the difference: $3\lceil S/2\rceil- S$.
\end{proof}

\subsection{Ending digits}

The beginning strings of integers are relatively sparse. What about the ending strings?

The last $k$ digits are repeated in a cycle which length is a multiple of 3. This is because for every three numbers only the last digit changes. The cycle repeats after the first occurrence of a number with $k$ zeros at the end.

\begin{lemma}
The last $k$ digits of integers in base 3/2 cycle with period $3^k$.
\end{lemma}

\begin{proof}
The number $3^k$ has $k$ zeros at the end. That means integers $x$ and $x+3^k$ have the same $k$ digits at the end. Now we need to show that there are no smaller cycles. Suppose two integers $x<y$ end with the same $k$ digits. Then $y-x$ ends with $k$ zeros. It follows that $y-x$ is divisible by $3^k$. This means the cycle must be a multiple of $3^k$.
\end{proof}

Interestingly, the endings of integers behave differently from the beginnings: at the end, any combination of the last several digits is possible.

\subsection{The largest and smallest integers with a given number of digits}

We look at the largest and smallest integers with a given number of digits expressed in base 3/2.

The smallest number $S_k$ with $k>0$ digits in base 3/2: 
\[0,\ 20,\ 210,\ 2100,\ 21010,\ 210110,\ 2101100,\ 21011000,\ 210110000, \ldots.\]

This is now sequence \seqnum{A304023}.

The values of these integers in base-10 form sequence \seqnum{A070885}:
\[1,\ 3,\ 6,\ 9,\ 15,\ 24,\ 36,\ 54,\ 81,\ 123,\ 186,\ 279,\ 420,\ 630,\ 945, \ldots.\]

The recursive formula for this sequence is: $a(n+1) = 3a(n)/2$, if $a(n)$ is even, and $a(n+1) = 3a(n+1)/2$ if $a(n)$ is odd \cite{JT}.

Correspondingly, the largest number $L_k$ with $k>0$ digits in base 3/2 is now sequence \seqnum{A304024}: 
\[2,\ 22,\ 212,\ 2122,\ 21222,\ 212212,\ 2122112,\ 21221112,\  \ldots.\]

Writing $L_k$ in base-10, we get a new sequence \seqnum{A304025}:
\[2,\ 5,\ 8,\ 14,\ 23,\ 35,\ 53,\ 80,\ \ldots.\]

These sequences are connected. The smallest number with $k+1$ digits is the largest number with $k$ digits plus 1: $S_{k+1} = L_k + 1$. In base 3/2, the largest number does not have zeros, and the smallest does not have twos except as the first digit. We prove the following necessary and sufficient condition.

\begin{lemma}\label{lemma:SL}
An integer is the largest integer with a given number of digits in base 3/2 if and only if it is represented without zeros and the last digit is 2.  Similarly, an integer is the smallest integer with a given number of digits in base 3/2 if and only if it ends in zero and is represented without twos except for the first digit.
\end{lemma}

\begin{proof}
If an integer is expressed without 0s and the last digit is 2, when we add 1 to the number, we will always be carrying over 2 to the next place. Then, we will keep on carrying all the way to the first digit and then carry over to have one more digit. On the other hand, if the last digit is not 2, we can always add 1 without changing the number of digits. If the last digit is 2 and the number has a zero in the middle, then the carry will stop at the first encountered zero and the number of digits will not increase. 

If an integer does not end in 0 or has an extra 2 that is not at the beginning, we can always subtract one from it to get to a smaller integer with the same number of digits. If it ends with 0 and does not have 2s other than the first digit, then subtracting 1 will decrease the number of digits.
\end{proof}

We can get $L_n$ directly from $S_{n+1}$ written in base 3/2: replace 21 at the beginning with 2 and 0 at the end with 2, then shift the rest of the digits up by 1. For example, 210110000 is divided into three groups: 21-011000-0. Replacing the first and the last group and shifting the middle, we get:  2-122111-2. The final result is 21221112.

If we remove the last digit from a $k$-digit smallest/largest integer, we get the smallest/largest even integer with $k-1$ digits. Thus, it is important to also look at even smallest/largest integers.

\subsection{The largest and smallest even integers with a given number of digits}

In this section we only consider even integers. Consider sequences $s_n$ and $l_n$ of the smallest and largest even positive integers with $n$ digits in base 3/2, where $n > 0$: $l_n+2 = s_{n+1}$. Sequence $s_n$ expressed in base 3/2 is now \seqnum{A303500} in the database and it starts as:
\[2,\ 21,\ 210,\ 2101,\ 21011,\ 210110,\ 2101100,\ 21011000,\ 210110001,\ \ldots.\]
Similarly, $l_n$ in base 3/2 is now \seqnum{A304272}:
\[2,\ 21,\ 212,\ 2122,\ 21221,\ 212211,\ 2122111,\ 21221112, 212211122, \ldots.\]

\begin{lemma}
In base 3/2, if we remove the last digit from integers $s_n$ and $l_n$, we get $s_{n-1}$ and $l_{n-1}$ correspondingly.
\end{lemma}

\begin{proof}
The smallest/largest integers form the leftmost/rightmost paths of the tree.
\end{proof}

Even integers are connected to all integers: either $s_n = S_n$ or $s_n = S_n+1$. Similarly, either $l_n = L_n$ or $l_n = L_n-1$. In any case, the following corollary follows from Lemma~\ref{lemma:SL}.

\begin{corollary}
The integers $l_n$ do not contain zeros when written in base 3/2. The integers $s_n$ do not contain twos except as the first digit when written in base 3/2.
\end{corollary}

Given that the smallest/largest even integer when written in base 3/2 with a given number of digits is a substring of the next one, we can create an infinite string representing all of the smallest/largest integers. We call this infinite string
\[2101100011010011010100110100101000\ldots.\]
corresponding to the smallest integers the \textit{ultimate smallest even integer}. Its digits are now sequence \seqnum{A304273}. We call this sequence of digits the \textit{evenberry} sequence.

Similarly, the \textit{evenmelon} sequence is the sequence of digits of the \textit{ultimate largest even integer} 
\[212211122121122121211221211212112\ldots\]
and is now sequence \seqnum{A304274}:

We can get the ultimate smallest even integer from the ultimate largest even integer by adding 2. That means, by replacing the first 2 with 21 and shifting all other digits down by 1.

The integer value of sequence $s_n$ in base-10 is now sequence \seqnum{A305498}:
\[2,\ 4,\ 6,\ 10,\ 16,\ 24,\ 36,\ 54,\ 82,\ \ldots.\]

The integer $\frac{3}{2}s_n$ is written in base 3/2 as $s_n$ in base 3/2 with zero at the end. Therefore, $\frac{3}{2}s_n$ is the smallest integer with $n+1$ digits: $\frac{3}{2}s_n = S_{n+1}$. The latter might not be even. Therefore, we can say that $s_{n+1} = \frac{3}{2}s_n$, if $s_n$ is divisible by 4, and that $s_{n+1} = \frac{3}{2}s_n + 1$ otherwise. Combining this together we get:
\[s_{n+1} = 2 \left \lceil\frac{3}{4}s_{n} \right \rceil. \]
Sequence $s_n$ is twice sequence \seqnum{A061419} in OEIS which is defined as: $a(n) = \lceil a(n-1)3/2\rceil$ with $a(1) = 1$.

Similarly, we can generate a recursive formula for the value of largest even number with $n$ digits in base 3/2, which is now sequence \seqnum{A305497}:
\[2,\ 4,\ 8,\ 14,\ 22,\ 34,\ 52,\ 80,\ \ldots.\]
The number $3/2\cdot l_n$ is an integer, and it is written in base 3/2 as $l_n$ with a zero at the end. That means the largest integer with $n+1$ digits is $3/2\cdot l_n+2$. This number might not be even. If it is odd, we need to subtract 1. Therefore, the largest even number is $2\cdot \lfloor (3/2\cdot l_n+2)/2\rfloor$. In other words:
\[l_{n+1} = 2 \left \lfloor\frac{3}{4}l_{n} \right \rfloor +2. \]

Sequence $l_n$ is twice sequence \seqnum{A006999} as we prove later. The description of sequence \seqnum{A006999} is the following: Partitioning integers to avoid arithmetic progressions of length 3. Given this does not provide enough detail, we provide a more detailed description following \cite{GPS}. Keep in mind that definition of \seqnum{A006999} has nothing to do with base 3/2.

Consider a greedy partition of non-negative integers into subsequences not containing a 3-term arithmetic progression. 

For example, the starting sequence $T_0$ is sequence \seqnum{A005836}: 0, 1, 3, 4, 9, 10, 12, and so on. This is the lexicographically earliest increasing sequence of nonnegative integers that contains no arithmetic progression of length 3. It is also the sequence of integers whose base 3 representation contains no 2.

We take the leftover numbers and build out of them the lexicographically earliest increasing sequence $T_1$ of nonnegative integers that contains no arithmetic progression of length 3: 2, 5, 6, 11, 14, 15, and so on. We take the leftover numbers and continue building $T_2$: 7, 8, 16, 17, 19, 20, and so on.

We define the characteristic sequence $a(n)$, so that $a(n) = k$  when $n \in T_k$. That is:
\[a(n) = 0,\ 0,\ 1,\ 0,\ 0,\ 1,\ 1,\ 2,\ 2,\ 0,\ 0, \ldots.\]

The latter is sequence \seqnum{A006997}. The database provides a formula for \seqnum{A006997}: $a(3n+k) = [ (3a(n)+k)/2 ]$, where $0 \leq k \leq 2$.

Sequence \seqnum{A006999} is defined through sequence \seqnum{A006997}: \seqnum{A006999}$(n) = \seqnum{A006997}(3^n-1)$. It is known that \seqnum{A006999}$(n)$ is the largest of the first $3^n$ terms of $a(n)$ \cite{GPS}. That means \seqnum{A006999}$(n)$ counts the number of sequences that appear in the greedy partition up to when integer $3^n$ is reached.

\begin{lemma}
Sequence $l_n$ is twice sequence \seqnum{A006999}.
\end{lemma}

\begin{proof}
The database provides a formula for \seqnum{A006997}: $a(3n+k) = \lfloor (3a(n)+k)/2 \rfloor$, where $0 \leq k \leq 2$. We can use this formula to make a recursive formula for \seqnum{A006999}$(n) = \seqnum{A006997}(3^n-1)$. First, we rewrite $\seqnum{A006997}(3^n-1)= \seqnum{A006997}(3^n-3+2) = \seqnum{A006997}(3(3^{n-1}-1)+2)$. Next, by the formula, $\seqnum{A006997}(3^n-1) = \lfloor (3\seqnum{A006997}(3^{n-1}-1)+2)/2 \rfloor$. Therefore,
\[\seqnum{A006999}(n) =  \lfloor (3\seqnum{A006999}(n-1)+2)/2 \rfloor = \lfloor 3\cdot 2\seqnum{A006999}(n-1)/4 \rfloor + 1.\]

If we denote $2\seqnum{A006999}(n)$ as $b(n)$, we get 
\[b(n) = 2\lfloor 3\cdot b(n-1)/4 \rfloor + 2,\]
which is the same recursion as the one for sequence $l_n$. After checking the initial term, we confirm that $l_n = 2\seqnum{A006999}(n)$.
\end{proof}

\section{Divisibility}\label{sec:divisibility}

We start with divisibility by powers of 3, see also \cite{JT}.

\begin{lemma}\label{lemma:div3}
An integer in base 3/2 has $k$ zeros at the end if and only if it is divisible by $3^k$.
\end{lemma}

\begin{proof}
Suppose 3-adic value of an integer $n$ is $k$. If we put $n$ dots in the units place, then, after exploding, we get 0 dots in the units place and $2n/3$ dots in the next place to the left. Continuing, we get exactly $k$ zeros at the end of the number $n$ in base 3/2.
\end{proof}

There is also a simple rule for divisibility by 5. It is similar to divisibility by 11 in base 10.

\begin{lemma}
The alternating sum of the digits of an integer in base 3/2, read from right to left, has the same remainder modulo 5 as the integer itself.
\end{lemma}

\begin{proof}
$3/2 \equiv -1 \pmod{5}.$
\end{proof}

\section{Sequences}\label{sec:sequences}

Here we discuss some sequences that we studied.

\subsection{Powers}

We start with powers of 3 written in base 3/2, which is now sequence \seqnum{A305658}:
\[1, 20, 2100, 212000, 210110000, 21202200000, 21200101000000, \ldots.\]

It follows from Lemma~\ref{lemma:div3} that $a(n) = 3^n$ in the above sequence of powers of 3 has $n$ zeros at the end.

The following is the sequence of powers of 2. Notice how it is similar to powers of 3:
\[ 1, 2, 21, 212, 21011, 212022, 21200101, 2101100202, 21202202121,  \ldots.\]
This is now sequence \seqnum{A305659}.
\begin{lemma}
An integer $3^n$ expressed in base $\frac{3}{2}$ is equal to $2^n$ in the same base with $n$ zeros appended at the end.
\end{lemma}

\begin{proof}
In order to
change a number from $2^n$ to $3^n$ in any base, we would need to
multiply $2^n$ by $\frac{3^n}{2^n} = \frac{3}{2}^n$, which, because we
are using base $\frac{3}{2}$, means that $n$ zeros would just be added
to the end of the number.
\end{proof}

\subsection{Look-and-say}

The base-10 look-and-say sequence, sequence \seqnum{A005150} in the OEIS, is the sequence of integers beginning as follows:
\[\textrm{Look-and-say: } 1,\ 11,\ 21,\ 1211,\ 111221,\ 312211,\ 13112221,\ 1113213211,\ \ldots.\]

It is a recursive sequence where the term $a(n+1)$ is defined by reading off the digits of $a(n)$, counting the number of digits in groups of the same digit. For example, 1211 is read off as ``one 1, one 2, then two 1s'' or 111221.

We study this sequence in base 3/2. The first five terms are the same as in the sequence base 10. For the next term, however, we need to read out three 1s, which in base 3/2 is 20 1s. Therefore, the sequence continues:
\[\textrm{Look-and-say}_{3/2}: 1,\ 11,\ 21,\ 1211,\ 111221,\ 2012211,\ 1210112221,\  \ldots.\]
This is now sequence \seqnum{A305660}.

\begin{lemma}
Each term in the Look-and-say sequence in base 3/2 has not more than 1 zero, 3 ones and 3 twos in a row.
\end{lemma}

\begin{proof}
Each string in this sequence can be divided into pairs of numbers: the first number counts how many of the second number are in the sequence. We call the first number the \textit{counting} number and the second number the \textit{what}-number. The neighboring what-numbers have to be different.

We proceed by induction. Observing several initial terms of the sequence, the base of induction holds. After that, by our induction hypothesis, the counting numbers can be only 1, 2, and 20.

Let us look at zeros. If the new term contains a zero as a what-number, it has to have 1 as a counting number in front and it has to have a counting number that does not start with zero after it. Therefore, the what-number that is zero must be isolated. Suppose we have 0 as a part of a counting number, then it has 2 before it. Also, by the induction hypothesis, it has to have 1 or 2 after it. 

Suppose the new term has at least 4 ones. Then one of these ones has to be a what-number. If there is a 1 after it, then it has to be followed by 0 or 2. If there is a 1 before it, this 1 is a counting number for our what-number 1. Before that 1 there only could be 0 or 2. Therefore, we cannot have more than 3 ones in a row.

Suppose the new term has at least 4 twos. Then one of these twos has to be a what-number. If there is a 2 after it as part of the counting number, then the counting number has to be either 2 or 20. In either case, the next digit must be not 2. If there is a 2 before the what-number 2, then the counting number is 2, and the digit before it has to be different from 2. Therefore, we cannot have more than 3 twos in a row.
\end{proof}

\subsection{Sorted Fibonacci}

John H.~Conway likes tweaking the Fibonacci rule to invent new sequences. He usually calls such sequences \textit{fibs}. We are following this tradition to emphasize that this is not a Fibonacci sequence.

The \textit{sorted Fibs} sequence is defined as follows.  To calculate the next term we add two previous terms and sort the digits in increasing order. In base-10 this sequence is \seqnum{A069638}:
\[0,\ 1,\ 1,\ 2,\ 3,\ 5,\ 8,\ 13,\ 12,\ 25,\ 37,\ 26,\ \ldots.\]

It is known that this sequence is periodic with the maximum value of 667.

We study analogues of this sequence in base 3/2. We start with the sorted Fibs sequence $f_n$ that have two initial values the same as in the Fibonacci sequence: $f_0=0$ and $f_1=1$. To calculate $f_{n+1}$ we add $f_{n-1}$ and $f_n$ in base 3/2 and sort the digits in increasing order. It follows that numbers in the sequence are represented with several ones followed by several twos.

Unlike base-10, the sequence is not periodic and grows indefinitely:
\[0,\ 1,\ 1,\ 2,\ 2,\ 12,\ 12,\ 112,\ 112,\ 1112,\ 1112,\ 11112,\ \ldots.\]
This sequence plays a special role in base 3/2 sorted Fibs. We call this sequence the \textit{Pinocchio} sequence. It is now sequence \seqnum{A305753}.

From now on we use the notation $\delta_k$ to denote a string of $k$ digits $\delta$ in a row. If there is only one digit we drop the index. The following lemma proves the pattern that can be seen in the Pinocchio sequence.

\begin{lemma}
In the Pinocchio sequence, we have $f_{2k} = f_{2k-1} = 1_{k-1}2$, where $k > 1$.
\end{lemma}

\begin{proof}
We prove this by induction. The base case holds. 

To calculate $f_{2k+1}$ we need to add $f_{2k-1}$ and $f_{2k}$, that is two numbers $1_{k-1}2$. The result is equal to $2_{k-1}4$ before the carries. Adding 2 and 2 means writing one and carrying two. As we continue carrying two to the beginning of the number we end up with $21_k$. After sorting we get the desired result. 

To calculate $f_{2k+2}$ we need to add $f_{2k}$ and $f_{2k+1}$. By the induction hypothesis $f_{2k+1} = f_{2k} + 10_k = f_{2k-1} + 10_k$. Using the previous calculation, $f_{2k} + f_{2k+1} = 21_k + 10_k = 201_k$. After sorting we get the desired result.
\end{proof}

The next interesting question is how the behavior of this sequence depends on the starting numbers. After the second number all terms of the sequence are sorted. From now on, we assume that we start with sorted numbers. Here are examples of two starting numbers when we end in the same pattern as above: (1,1), (2,112), or (1,12). 

However, not all starting numbers end in the Pinocchio sequence. Starting with 2 and 22 we get 2, 22, 112, 122, 1122, 122, 122, 112, 1122, 1122, 112, 1122, and so on. It becomes periodic with a period-3 cycle: 112, 1122, 1122. 

Our goal is to show that for any starting terms the sequence eventually turns either into the Pinocchio sequence or into the 3-cycle above. 

When we add two sorted numbers, we can represent the result as the sum $1_a2_b3_c4_d$ before we do the carries. The following lemma describes the result after the carries and sorting.

\begin{lemma}\label{lemma:carry}
Given the string $1_a2_b3_c4_d$, after performing the carries and sorting, the resulting string is the following:
\begin{enumerate}
\item $a >0$ and $d > 1$: $1_{c+1}2_d$.
\item $a = 0$ and $d > 1$: $1_{c+2}2_{d-1}$.
\item $d=1$: $1_{b+1}2_{c+1}$.
\item $c > 0$ and $d = 0$: $1_{b}2_{c}$.
\item $c=0$ and $d =0$: $1_{a}2_{b}$. 
\end{enumerate}
\end{lemma}

\begin{proof}
We start by assuming $a >0$ and $d > 1$. After the carries we get:
$20_{a-1}20_{b}1_c2_{d-2}01.$
Then after sorting we get: $1_{c+1}2_d$.

If $a = 0$ and $d > 1$, after the carries we get $210_b1_c2_{d-2}01.$
Then after sorting we get $1_{c+2}2_{d-1}$.

If $d=1$, after the carries we get $20_a1_b2_c1$. When sorted, we get $1_{b + 1}2_{c + 1}$.

If $d = 0$ and $c > 0$ after carries we get $20_a1_b2_{c-1}0$. After sorting we get $1_b2_c$.

Finally, if $c=0$ and $d =0$. There are no carries and after sorting the result is the same: $1_a2_b$.
\end{proof}

Numbers $a$, $b$, $c$, and $d$ play a big role in the coming proofs. For this reason, we want to associate them with every term of the sequence. That is, $a_n$, $b_n$, $c_n$, and $d_n$ correspond to the sum of $f_{n-2}$ and $f_{n-1}$ before carry. In our assumption, all terms of the sequence are sorted. Let $z_n$ be the number of 2s in the $n^{th}$ entry, and let $y_n$ be the number of 1s in the $n^{th}$ entry.

Integers $a_n$, $b_n$, $c_n$, and $d_n$ give us some information about $f_{n-2}$ and $f_{n-1}$. For example, we know the minimum of the number of twos:
\[\min\{z_{n-2},z_{n-1}\} = d.\]
For the maximum there are two possibilities:
\[\max\{z_{n-2},z_{n-1}\} = c+d \quad \text{ or} \quad \max\{z_{n-2},z_{n-1}\} = b+c+d.\]

The second situation happens when one of the numbers is $1_c2_d$ and the other is $1_a2_{b+c+d}$. 
 
We can also estimate the total number of digits:
\[c+d \leq \min\{y_{n-2}+z_{n-2},y_{n-1}+z_{n-1}\} \leq b+c+d\]
and 
\[\max\{y_{n-2}+z_{n-2},y_{n-1}+z_{n-1}\} = a+b+c+d.\]

Every term in the sorted Fibs sequence, except for the first few terms, has at least one 1 and one 2 as the following corollary explains.

\begin{corollary}
For a sorted Fibs sequence that starts with sorted strings, if $n \geq 2$, then $z_n > 0$. Also, if $n \geq 4$, then $y_n > 0$.
\end{corollary}

\begin{proof}
The only case in the list in Lemma~\ref{lemma:carry} when the resulting number of 2s is zero is the last one when $b = c= d = 0$. This case is impossible as we are summing up two non-zero numbers, and the last digit before carry must be greater than 1. 
When there is at least one 2 in each number, then the last digit of the next number in the sequence will be $2+2\implies 1$, so there must be a 1 in the number. 
\end{proof}

We can bound sequence $z_n$ of the number of twos.

\begin{lemma}
If $n \geq 4$, then $z_n \leq \max\{z_{n-1},z_{n-2}\}$.
\end{lemma}

\begin{proof}
If $n \geq 4$, then both $f_{n-1}$ and $f_{n-2}$ have twos. That means $d_n >0$. Therefore, from Lemma~\ref{lemma:carry}, we have $z_n$ is one of: $d_n$, $d_n-1$, or $c_n+1$. In either case, $z_n \leq c_n+d_n$. On the other hand, one of the previous numbers has at least $c_n+d_n$ twos. 
\end{proof}

Let us denote the maximum number of twos in two consecutive terms $f_n$ and $f_{n+1}$ as $m_n$: $m_n = \max\{z_n,z_{n+1}\}$. From the previous lemma it follows that $m_{n+1} \leq m_n$, for $n \geq 5$. As our sequence is infinite, it follows that $m_n$ stabilizes. As we are interested at the eventual behavior of sorted Fibs, we proceed by studying sequences where $m_n$ is fixed and equal to $M$. We call such sequences \textit{M-stable}.

Let us assume that sequence $f_n$ is M-stable. One example, is the Pinocchio sequence, starting from index 3, where $z_n = 1$.

\begin{lemma}[Fluctuation Lemma]
An M-stable sequence for $n>0$ is either a subsequence of the Pinocchio sequence, or, for that sequence, $z_n$ can have only two values: $M$ and 1.
\end{lemma}

\begin{proof}
We consider cases depending on the behavior of the number of twos:

\begin{enumerate}
\item $z_n =M > 1$ for any $n$;
\item $z_n = M = 1$ for any $n$;
\item $z_n$ varies.
\end{enumerate}

Case 1. Suppose the number of twos does not change and is more than 1: $z_n =M > 1$. Given that $z_{n-2} = z_{n-1} =M$, we get  $c_n=0$ and $d_n =M$. From the fact that $z_n =M >1$, it follows that this could only be case 1 from Lemma~\ref{lemma:carry} and $f_n = 12_d$. Similarly, $f_{n+1} = 12_d$. Summing them up, we get $f_{n+2} = 1112_{d-1}$: the number of twos is reduced, which is a contradiction. 

Case 2. Suppose the number of twos does not change and is 1: $z_n =M = 1$.
The sequence of the number of ones can start as:

\begin{itemize}
\item $y_{n-2}=a+b$, $y_{n-1} = b$, where $a > 0$. Then $y_n = b+1$.
\item $y_{n-2}=b$, $y_{n-1} = b+1$.
\item $y_{n-2}=b$, $y_{n-1} = a+b$, where $a > 1$. Then it continues as $y_n = b+1$, $y_{n+1} = b+2$.
\end{itemize}

In all cases we get into the Pinocchio sequence. 

Case 3. Suppose $z_{n-2} = M > z_{n-1}$. Then $c_n > 0$ and $d_n = z_{n-1}$. Then $z_{n} =M$ to guarantee M-stability. Therefore, $z_n \geq c_n + d_n$. On the other hand, from Lemma~\ref{lemma:carry}, $z_n$ must be either $d$, $d-1$, or $c+1$. The only possibility is that $d_n = c_n+d_n$ and $d_n=1$. Therefore, $z_{n-1} =1$. Thus we showed that if $z_n \neq M$, then $z_n = 1$.
\end{proof}

Next we want to show that if an M-stable sequence has a varying number of twos and $M >2$, then the number of twos strictly alternates between $M$ and 1. We already know that such a sequence $z_n$ cannot have two ones in a row. What is left to show that it does not have two $M$s in a row. We show that if it does have two $M$s in a row, then $M$ does not vary.

\begin{lemma}
In an M-stable sorted Fibs sequence, suppose that $z_n=z_{n+1}=M>2$. Then $z_{k}=M$, for $k > n+1$.
\end{lemma}

\begin{proof}
The sum of the $n^{th}$ and $(n+1)^{th}$ entry is of the form $1_a 2_b 4_M$, as the terms $f_n$ and $f_{n+1}$ have the same number of twos. If $M>1$, then $z_{n+2}=M$ or $z_{n+2}=M-1$. But $M-1 \neq 1$, so this case is impossible.
Therefore, $z_{n+2}=M$. Similarly, $z_{k}=M$, for all consecutive $k$.
\end{proof}

We showed that if $M > 2$ and varies, then the sequence $z_n$ alternates: $\ldots$, $M$, 1, $M$, 1, $\ldots$. Now we show that sequence $z_n$ cannot have subsequence 1, $M > 1$, 1.

\begin{lemma}
The case $z_{n-2} = z_n = 1$ and $z_{n-1} > 1$ is impossible.
\end{lemma}

\begin{proof}
If $z_{n-2} = 1$ and $z_{n-1} > 1$, then $d_n=1$, and $c_n > 0$. That means we are in the third case in Lemma~\ref{lemma:carry}, and $z_n = c_n+1$, contradiction. 
\end{proof}

This excludes the alternating case for $M > 2$. It also means, that if $M=2$, then we still cannot have 1, 2, 1 in the sequence $z_n$.

Now we are ready for our classification theorem.

\begin{theorem}
Any sorted Fibs sequence eventually turns into either the Pinocchio sequence or the 3-cycle 112, 1122, 1122.
\end{theorem}

\begin{proof}
We already know that the sequence either turns into the Pinocchio sequence or, starting from some $n$, $m_n=2$ and $z_n$ sequence varies and does not contain a subsequence 1, 2, 1. Therefore, we can find integer $N$, such that $z_{N-2} = z_{N-1} = 2$ and $z_N=1$. That means numbers $a_N$, $b_N$, $c_N$, $d_N$ should correspond to the second case in Lemma~\ref{lemma:carry}. That is, $a_N=0$. Therefore $y_{N-2}=y_{N-1}$. Also $c_N=0$, thus $a_N=112$. For the next step we get $b_{N+1}=c_{N+1}=d_{N+1}=1$. Therefore, the next number is 1122. And we got into our cycle.
\end{proof}

\subsection{Reverse sorted Fibs}

The \textit{reverse sorted Fibs} sequence $r_n$ in base 3/2 is defined as follows: To calculate $r_{n+1}$, we add $r_{n-1}$ and $r_n$ in base 3/2 and sort the digits in decreasing order, ignoring zeros. It follows that numbers in the sequence are represented with several twos followed by several ones. The base-10 analog without discarding zeros exists in the OEIS database: \seqnum{A237575} Fibonacci-like numbers with nonincreasing positive digits.

We call the sequence that starts similar to Fibonacci sequence with $r_0=0$ and $r_1=1$, the \textit{proper reverse sorted Fibs}. Here are several terms of the proper reverse sorted Fibs: 0, 1, 1, 2, 2, 21, 21, 221, 2211, 221, 221, 2211, 221, 221, 2211, $\ldots$. This sequence becomes cyclic, starting from $r_7$.

As in the previous section, we use the notation $\delta_k$ to denote a string of $k$ digits $\delta$ in a row. If there is only one digit we drop the index.

We want to study the eventual behavior of the reverse sorted Fibs depending on the starting terms. By computational experiments, we found a series of 3-cycles that such a sequence can turn into: 
\[2_k1,\ 2_k1,\ 2_k1_2,\]
where $k > 1$.

We also found a sequence growing indefinitely: 
\[2_k1_2,\ 2_k1_2,\ 2_{k+1}1_2,\ 2_{k+1}1_2,\ 2_{k+2}1_2,\ 2_{k+2}1_2,\]
and so on, where $k > 1$. We were surprised by the fact that the sorted Fibs and the reverse sorted Fibs were so similar. They both have exactly one sequence that grows indefinitely. To emphasize this analogy, we reversed the word Pinocchio to call this growing reverse Fibs sequence the \textit{Oihcconip} sequence.

Our goal is to prove that the eventual behavior of a reverse sorted Fibs sequence must be one of the 3-cycles or the tail of the Oihcconip sequence.

Let us compute a sum of two sorted numbers before carry. It could be of the form $2_a 1_b 3_c 2_d$, or it could be of the form $2_a 4_b 3_c 2_d$, where some of the indices might be zero. The first case happens when the number of ones of one of the numbers is greater or equal to the number of digits of the other number. Alternatively, we can say that the cases depend on whether the positions of twos overlap or not.

For the first case, we have $2_a 1_b 3_c 2_d= 2_a 1_{b-1} 3 2_{c-1} 0 2_d=2_a 3 0_{b-1} 2_{c-1} 0 2_d=2 1_{a}0 0_{b-1} 2_{c-1} 0 2_d$, which sorts to $2_{c+d}1_{a}$, assuming $b, c \geq 1$. If $c = 0$, then there is no carry, and the result is $2_{a+d}1_b$. When $b = 0$ and $c > 0$ we get $2_a 3_c 2_d = 2_{a-1} 4 2_{c-1} 0 2_d=2 1_{a} 2_{c-1} 0 2_d$, which sorts to $2_{c+d}1_{a}$. We summarize this into Table~\ref{tab:case1}.

\begin{table}[htb!]
\begin{center}
\begin{tabular}{ |c |r |r|}
\hline
  case    & before carry & after carry \\
    \hline
 $c>0$  & $2_a 1_b 3_c 2_d$ & $2_{c+d}1_{a}$ \\
 $c=0$  & $2_a 1_b 2_d$ &$2_{a+d}1_b$ \\

\hline
\end{tabular}
\end{center}
\caption{Case 1.}\label{tab:case1}
\end{table}

For the second case, we have $2_a 4_b 3_c 2_d=2_a 4_{b-1} 6 2_{c-1} 0 2_d=2_a 8 2_{b-2} 0 2_{c-1} 0 2_d=2 1 0_{a} 22_{b-2} 0 2_{c-1} 0 2_d$, which sorts to $2_{b+c+d-1} 1$, assuming $b \geq 2$ and $c \geq 1$.

For this case, we do not need to check $b=0$, as it is a subcase of the first case. We do need to check the cases $c=0$ and/or $b=1$. The summary is in Table~\ref{tab:case2}.

\begin{table}[htb!]
\begin{center}
\begin{tabular}{ |c|r |r |}
\hline
  case   & before carry & after carry \\
    \hline
$b>1$ $c>0$ & $2_a 4_b 3_c 2_d$   & $2_{b+c+d-1} 1$ \\
$b>1$ $c=0$ & $2_a 4_b 2_d$  &  $2_{b+d-1}1_{2}$ \\
$b=1$ $c>0$ & $2_a 4 3_c 2_d$   & $2_{c+d}1$ \\
$b =1$ $c=0$ & $2_a 4 2_d$   &  $2_{d+1}1_{a+1}$  \\
\hline
\end{tabular}
\end{center}
\caption{Case 2.}\label{tab:case2}
\end{table}

Combining two cases together and adding a column for previous terms we get Table~\ref{tab:twocases}.

\begin{table}[htb!]
\begin{center}
\begin{tabular}{ |c|r |r | r|}
\hline
  case   & before carry & after carry & previous numbers\\
    \hline
Case 1.  $c>0$  & $2_a 1_b 3_c 2_d$ & $2_{c+d}1_{a}$ & $2_a1_{b+c+d}$ and $2_c1_d$\\
Case 1.  $c=0$  & $2_a 1_b 2_d$ &$2_{a+d}1_b$ & $2_a1_{b+d}$ and $1_d$\\
Case 2.  $b>0$ $c>0$ & $2_a 4_b 3_c 2_d$   & $2_{b+c+d-1} 1$ & $2_{a+b}1_{c+d}$ and  $2_{b+c}1_d$; $2_{a+b+c}1_{d}$ and $1_{c+d}$; \\
Case 2. $b>1$ $c=0$ & $2_a 4_b 2_d$  &  $2_{b+d-1}1_{2}$ & $2_{a+b}1_{d}$ and $2_b1_d$\\
Case 2. $b =1$ $c=0$ & $2_a 4 2_d$   &  $2_{d+1}1_{a+1}$ &  $2_{a+1}1_{d}$ and $21_d$\\
\hline
\end{tabular}
\end{center}
\caption{Two cases together}\label{tab:twocases}
\end{table}

We start with discussing the number of twos.

\begin{lemma}
Starting from $n>1$, each element of a reverse sorted Fibs sequence contains 2.
\end{lemma}

\begin{proof}
If the sum of two terms has a carry, the result has to have a two. If it does not, the last digits of the previous terms have to be ones, and the result has to have a two.
\end{proof}

That means we can remove some cases from Table~\ref{tab:twocases} to generate the new Table~\ref{tab:leftovercases}.

\begin{table}[htb!]
\begin{center}
\begin{tabular}{ |c|c|r |r | r|}
\hline
line &  case   & before carry & after carry & previous numbers\\
    \hline
L1 & $a>0$ $c>0$  & $2_a 1_b 3_c 2_d$ & $2_{c+d}1_{a}$ & $2_a1_{b+c+d}$ and $2_c1_d$\\
L2 & $c>0$ $b > 0$ & $2_a 4_b 3_c 2_d$   & $2_{b+c+d-1} 1$ & $2_{a+b}1_{c+d}$ and $2_{b+c}1_d$ \\
L3 & $b>1$ $c=0$ & $2_a 4_b 2_d$  &  $2_{b+d-1}1_{2}$ & $2_{a+b}1_{d}$ and $2_b1_d$\\
L4 & $b =1$ $c=0$ & $2_a 4 2_d$   &  $2_{d+1}1_{a+1}$ &  $2_{a+1}1_{d}$ and $21_d$\\
\hline
\end{tabular}
\end{center}
\caption{Leftover cases}\label{tab:leftovercases}
\end{table}

We can see that now the result contains at least one one. Therefore, we can assume that $d > 0$. Now we are ready to study the eventual behavior of reverse sorted Fibs sequences.

We consider each line in Table~\ref{tab:leftovercases} separately starting from Line 2.

\begin{lemma}
If we start in Line 2, then we end up in a cycle sequence.
\end{lemma}

\begin{proof}
Recall that $b,c,d > 0$. The sum of $r_0$ and $r_1$ before carry is $2_a 4_b 3_c 2_d$ forcing $r_2 = 2_{b+c+d-1} 1$.

We consider two cases based on the order of the previous terms.

\textbf{Case 1.} $r_0=2_{a+b}1_{c+d}$ and $r_1= 2_{b+c}1_d$.

If $r_1 = r_2$, that is $d=1$, we are in a cycle sequence. If not, then $r_3= 2_{b+c+d-1} 1 =r_2$, and we get into a cycle sequence anyway.

\textbf{Case 2.} $r_0= 2_{b+c}1_d$ and $r_1=2_{a+b}1_{c+d}$.

As $c+d > 1$, for the next number we are in Line 2 again, and $r_3 = 2_{b+c+d-1} 1$. We end up in a cycle sequence again.

In any case $r_2$ is in a cycle.
\end{proof}

Now we go to Line 3.

\begin{lemma}
If we start in Line 3, then we end up in a cycle if $d \neq 2$. We also end in a cycle if $r_0=21_2$ and $r_1=2_{a+1}1_2$. Otherwise, we end in the Oihcconip sequence.
\end{lemma}

\begin{proof}
Recall that $b,d > 0$. The sum of $r_0$ and $r_1$ before carry is $2_a 4_b 2_d$.

We consider two cases based on the order of the previous terms.

\textbf{Case 1.} Suppose $r_0=2_{a+b}1_{d}$ and $r_1=2_b1_d$.

Then $r_2 =  2_{b+d-1} 1_2$. We consider cases.

\begin{itemize}
\item If $b=d=1$, then $r_1 = 21$, $r_2 = 21_2$, $r_3 = 2_21$, and $r_4 = 2_21$. We are in a cycle sequence starting from $r_3$.
\item If $d \neq 2$ and $b+d > 2$, then $r_1$ and $r_2$ correspond to Line 2 and the next term, $r_3$ must be in a cycle sequence.
\item If $d = 2$ and $b>1$, then $r_1=2_b1_2$ and $r_2 =  2_{b+1} 1_2$ and we are in the Oihcconip sequence starting from $r_1$.
\item If $d = 2$ and $b=1$, then $r_1=2_11_2$, $r_2 =  2_2 1_2$, and $r_3 = 2_31_2$. We are in the Oihcconip sequence starting from $r_2$.
\end{itemize}

\textbf{Case 2.} Suppose $r_0=2_b1_d$ and $r_1=2_{a+b}1_{d}$. We can assume that $a>0$, otherwise we are covered by the previous case. 

Then $r_2 =  2_{b+d-1} 1_2$. We consider cases.

\begin{itemize}
\item If $d \neq 2$, then $r_1$ and $r_2$ correspond to Line 2 and the next term, $r_3$, must be in a cycle sequence.
\item If $d=2$ and $b \geq 1$, we have $r_1 = 2_{a+b}1_2$ and $r_2 = 2_{b+1}1_2$. As $b> 1$, we are now in Line 3, and $r_3 =  2_{b+2}1_2$. Therefore, we are in the Oihcconip sequence starting from $r_2$.
\end{itemize}
\end{proof}

Now we go to Line 1.

\begin{lemma}
If we start in Line 1, then we end up in the Oihcconip sequence or a cycle.
\end{lemma}

\begin{proof}
The two previous terms are $2_a1_{b+c+d}$ and $2_c1_d$. Also, $a, c, d > 0$. We already know that we must eventually enter a cycle or the Oihcconip sequence if we get to Line 2 or 3.

We consider two cases based on the order of the previous terms. 

\textbf{Case 1.} Suppose $r_0 = 2_a1_{b+c+d}$, $r_1 = 2_c1_d$, and  $r_2 = 2_{c+d}1_{a}$. Now we see which line corresponds to $r_1$ and $r_2$.

\begin{itemize}
\item If $a < d$ we are on Line 2.
\item If $a = d$, and $c > 1$, we are on Line 3.
\item If $a=d$, and $c=1$, we have $r_2 = 2_{d+1}1_{d}$ and we get $r_3 = 2_{d+1}1_{d+1}$ putting us on Line 2.
\item If $d < a < c+d$, we are on Line 3.
\item If $a \geq c+d$, we get $r_3=2_{c+d} 1_{c+d}$. For the next step we have subcases. a) If $c+d\leq a<2c+2d$, we are on Line 2. b) If $a \geq 2c+2d$, then we are on Line 1 and $r_4=2_{2c+2d}1_{c+d}$. After that we get to Line 3 again.
\end{itemize}

\textbf{Case 2.} Suppose $r_0 = 2_c1_d$ and $r_1 = 2_a1_{b+c+d}$, and  $r_2 = 2_{c+d}1_{a}$.

\begin{itemize}
\item If $a > b+c+d$, we are on Line 2.
\item If $a = b+c+d$, we are on Line 3.
\item If $b < a < b+c+d$, we are on Line 2.
\item If $a \leq b$, then we are on Line 1 and $r_3=2_{c+d+a}1_a$. For the next step we are on Line 3.
\end{itemize}
\end{proof}

Finally we go to Line 4.

\begin{lemma}
If we start in Line 4, then we end up in the tail of the Oihcconip sequence or a cycle.
\end{lemma}

\begin{proof}
The two previous terms are $2_{a+1}1_{d}$ and $21_d$. Also $d > 1$. We already know that we must eventually enter a cycle or the Oihcconip sequence if we get to Line 1, 2 or 3.

We consider two cases based on the order of the previous terms. Now we see which line corresponds to $r_1$ and $r_2$.

\textbf{Case 1.} Suppose $r_0= 2_{a+1}1_{d}$, $r_1 = 21_d$, then $r_2 = 2_{d+1}1_{a+1}$. 

\begin{itemize}
\item If $a \geq d$, we are on Line 1.
\item If $a = d-1$, then $r_2 = 2_{d+1}1_{d}$ and $r_3 = 2_{d+1}1_{d+1}$. Now we get on Line 3.
\item If $a < d-1$, we are on Line 3.
\end{itemize}

\textbf{Case 2.} Suppose $r_0 = 21_d$ and $r_1 = 2_{a+1}1_{d}$, and  $r_2 = 2_{d+1}1_{a+1}$.

\begin{itemize} 
\item If $a+1 \neq d$, we are on Line 2.
\item If $a+1 = d \neq 1$, then we are on Line 3.
\item If $a+1 =d=1$, then $r_1=21$ and $r_2=2_21$. Then $r_3= 2_21_2$ and we get into a cycle.
\end{itemize}
\end{proof}

We can summarize the results into the following theorem.

\begin{theorem}
We always end in a cycle or the tail of the Oihcconip sequence.
\end{theorem}

\section{Acknowledgements}

We are grateful to PRIMES STEP for allowing us to do this research. We also want to thank Prof.~James Propp for helpful suggestions
and Glen Whitney for pointing us to the tree.

\bigskip
\hrule
\bigskip

\noindent 2010 {\it Mathematics Subject Classification}: Primary 11B99; Secondary 11A99.

\noindent \emph{Keywords:} fractional bases, base 3/2, sorted Fibonacci.

\bigskip
\hrule
\bigskip

\noindent (Concerned with sequences 
\seqnum{A005150}, 
\seqnum{A005428},
\seqnum{A005836},
\seqnum{A006997},
\seqnum{A006999}, 
\seqnum{A024629}, 
\seqnum{A061419}, 
\seqnum{A069638},
\seqnum{A070885},
\seqnum{A073941},
\seqnum{A081848},
\seqnum{A237575},
\seqnum{A246435},
\seqnum{A303500},
\seqnum{A304023},  
\seqnum{A304024},
\seqnum{A304025},
\seqnum{A304272},
\seqnum{A304273},
\seqnum{A304274},
\seqnum{A305497},
\seqnum{A305498},
\seqnum{A305658},
\seqnum{A305659},
\seqnum{A305660}, and
\seqnum{A305753})
\end{document}